\documentclass[12pt,reqno]{amsart}

\usepackage{amssymb,amsmath,amsthm,amscd,latexsym,amsfonts}
\usepackage{mathtools}
\usepackage[T1]{fontenc}
\usepackage{graphicx,epstopdf}
\usepackage{graphicx}
\usepackage{xcolor}
\usepackage{cite}
\usepackage{float}
\usepackage[a4paper,top=3.0cm, bottom=3.0cm, left=3.0cm, right=3.0cm]{geometry}

\newtheorem{thm}{Theorem}

\newtheorem{lemma}{Lemma}

\newtheorem{rk}{Remark}

\numberwithin{equation}{section} 

\newcommand{\bea}{\begin{eqnarray}}
	\newcommand{\eea}{\end{eqnarray}}

\newcommand{\R}{\mathbb{R}}
\newcommand{\C}{\mathbb{C}}

\def\r{\rho}

\def\R{\mathbb{R}}

\def\C{\mathbb{C}}

\begin{document}
	\title [Dynamical systems of M\"obius transformation]
	{Dynamical systems of M\"obius transformation: real, $p$-adic and complex variables}
	
	\author{E.T. Aliev,  U.A. Rozikov}
\address{E.T. Aliev \\ Namangan Institute of Engineering Technology, 7, Kosonsoy street, 160115, Namangan, Uzbekistan.} \email
{aliev-erkinjon@mail.ru}
	
	\address{ U.A. Rozikov$^{a,b,c}$\begin{itemize}
			\item[$^a$] V.I.Romanovskiy Institute of Mathematics,  9, Universitet str., 100174, Tashkent, Uzbekistan;
			\item[$^b$] New Uzbekistan University,
			54,  Mustaqillik ave.,    100007,  Tashkent, Uzbekistan;
			\item[$^c$] National University of Uzbekistan,  4, Universitet str., 100174, Tashkent, Uzbekistan.
	\end{itemize}}
	\email{rozikovu@yandex.ru}
	
	\begin{abstract} In this paper we consider function $f(x)={x+a\over bx+c}$, (where $b\ne 0$, $c\ne ab$, $x\ne -{c\over b}$) on three fields: the set of real,  $p$-adic and complex numbers. We study dynamical systems generated by this function on each field separately and give some comparison remarks.
\begin{itemize}		
	\item	 For real variable case we show that the real dynamical system of the function depends on the parameters $(a,b,c)\in \mathbb R^3$. Namely, we classify the parameters to three sets and prove that: for the parameters from first class each point, for which the trajectory is well defined, is a periodic point of $f$; for  the parameters from second class any trajectory (under $f$) converges to one of fixed points (there may be up to two fixed points); for the parameters from third class any trajectory is dense in $\mathbb R$.  
	
\item For the $p$-adic variable we give a review of known results about dynamical systems of function $f$. Then using a recently developed method we give simple new proofs of these results and prove some new ones related to trajectories which do not converge. 

\item For the complex variables we give a review of known results.
\end{itemize} 	
	\end{abstract}
	\maketitle
	
	
	\section{Introduction}
Dynamical systems generated by rational functions appear in several problems of natural sciences (see, for example \cite{Be}, \cite{Rbp}, \cite{Rpd} and references therein). Depending on the nature of the considered problem in biological or physical systems a rational function can be considered on the field of real, $p$-adic or complex numbers.  

 In this paper we consider rational function $f(x)={x+a\over bx+c}$, where $b\ne 0$, $c\ne ab$, $x\ne -{c\over b}$, on each above mentioned fields. The main goal is to study iterations of the function to each point of the field. Recall that the sequence of iterations of the function to a point is called a discrete-time dynamical system or trajectory of the point. The main problem in the theory of dynamical systems is to describe the set of all limit points of each trajectory.
 
 On the field of complex numbers this problem is well studied (see e.g. \cite{Be}). In the best of our knowledge there is no any paper devoted to complete analysis of the dynamical system of $f$ on the field of real numbers. In this paper we will completely study such real dynamical systems. 
 
In \cite{Avetisov},\cite{Albeverio},\cite{Dubischer},\cite{xrennikov3},\cite{xrennikov4} some applications of $p$-adic dynamical systems to biological and physical systems are given. Rational  $p$-adic dynamical systems appear when studying $p$-adic Gibbs measures \cite{ganikhodjaev},\cite{khamraev},\cite{mukhamedov},\cite{Mukhamedov2}.

In \cite{Mukhamedov3} the behavior of a trajectories under the rational function $f$ on the complex $p$-adic field $\mathbb C_p$ are studied. Siegel disks and basin of attractors are described.

 In this paper we give simple new proofs of these results. Moreover, we give some new results related to the limit points when the trajectory does not converge.  
Moreover, we give a brief review about the dynamical systems of $f$ given on the set of complex numbers. 

\section{Preliminaries}
Let us give main definitions and problems related to dynamical systems (see \cite{HK}, \cite[Chapter 1]{Rmb}).
\subsection{Dynamical systems}
  Let $f$ be a function defined on a topological space $X$.

 Denote $f^n(x)$, meaning $f$ is applied to $x\in X$ iteratively $n$ times.
	
Let $A$ be a subset of $X$. Then $f(A)=\{f(x): x\in A\}$.
	
If $f(A)\subset A$, then $A$ is called an invariant set under function $f$.
	
The point $x\in X$ is called a fixed point for $f$ if $f (x) = x.$

The point $x\in X$ is a periodic point of period $m$ for $f$ if $f^m(x)=x.$ The least positive $m$ for which $f^m(x)=x$ is called the prime period of $x.$

 We denote
the set of all periodic points of (not necessarily prime) period $m$ by ${\rm Per}_m(f)$, and the set of all fixed points by ${\rm Fix}(f).$

For given topological space $X$, $x^{(0)}\in X$ and $f:X\to X$ the discrete-time dynamical system is defined as
\begin{equation}\label{0a2}
	x^{(0)},\ \ x^{(1)}= f(x^{(0)}), \ \ x^{(2)}=f^{2}(x^{(0)}),\ \  x^{(3)}= f^{3}(x^{(0)}),\dots
\end{equation}

{\bf The main problem:} for a given dynamical system is to
describe the limit points of $\{x^{(n)}\}_{n=0}^\infty$ for
arbitrary given $x^{(0)}$.\\

{\bf Aim of the paper:} In this paper we are interested to the dynamical system generated by M\"obius transformation, which is a rational map of the form
$$f(x)={x+a\over bx+c}, \ \ c-ab\ne 0.$$
Here for $X$ we consider three cases: the field of real numbers, p-adic numbers and complex numbers. For each case we collect known results about dynamical systems of the M\"obius transformation and give our new results. Moreover,  we present some comparison of methods and results.

\subsection{$p$-adic numbers}

Let $\mathbb{Q}$ be the field of rational numbers. The greatest common
divisor of the positive integers $n$ and $m$ is denotes by
$(n,m)$. For a fixed prime number $p$, every rational number $x\neq 0$ can be represented as $x=p^r\frac{n}{m}$, where $r,n\in\mathbb{Z}$, $m$ is a
positive integer, $(p,n)=1$, $(p,m)=1$.

The $p$-adic norm of $x$ is defined as
$$
|x|_p=\left\{
\begin{array}{ll}
p^{-r}, & \ \textrm{ for $x\neq 0$},\\[2mm]
0, &\ \textrm{ for $x=0$}.\\
\end{array}
\right.
$$
It has the following properties:

1) $|x|_p\geq 0$ and $|x|_p=0$ if and only if $x=0$,

2) $|xy|_p=|x|_p|y|_p$,

3) the strong triangle inequality
$$
|x+y|_p\leq\max\{|x|_p,|y|_p\}.
$$
More precisely 

3.1) if $|x|_p\neq |y|_p$ then $|x+y|_p=\max\{|x|_p,|y|_p\}$,

3.2) if $|x|_p=|y|_p$, $p\neq 2$ then
$|x+y|_p\leq |x|_p$, and $|x+y|_2\leq {1\over 2} |x|_2$.

The completion of $\mathbb{Q}$ with  respect to $p$-adic norm defines the
$p$-adic field which is denoted by $\mathbb{Q}_p$ (see \cite{Ko}).

The algebraic completion of $\mathbb{Q}_p$ is denoted by $\C_p$ and it is
called {\it complex $p$-adic numbers}. Note that (see \cite{Ko})
$\C_p$ is algebraically closed, an infinite dimensional vector space over $\mathbb{Q}_p$.

 For any $a\in\C_p$ and
$r>0$ denote
$$
U_r(a)=\{x\in\C_p : |x-a|_p<r\},\ \ V_r(a)=\{x\in\C_p :
|x-a|_p\leq r\},
$$
$$
S_r(a)=\{x\in\C_p : |x-a|_p= r\}.
$$


\subsection{Dynamical systems in $\mathbb C_p$.}
To define a dynamical system we consider $X=\mathbb C_p$ and a function  $f: x\in U\to f(x)\in U$ (see \cite{PJS}).


A fixed point $x_0$ is called an {\it
	attractor} if there exists a neighborhood $U(x_0)$ of $x_0$ such
that for all points $x\in U(x_0)$ it holds
$\lim\limits_{n\to\infty}f^n(x)=x_0$. If $x_0$ is an attractor
then its {\it basin of attraction} is
$$
\mathcal A(x_0)=\{x\in \C_p :\ f^n(x)\to x_0, \ n\to\infty\}.
$$
A fixed point $x_0$ is called {\it repeller} if there  exists a
neighborhood $U(x_0)$ of $x_0$ such that $|f(x)-x_0|_p>|x-x_0|_p$
for $x\in U(x_0)$, $x\neq x_0$.

Let $x_0$ be a fixed point of a
function $f(x)$.
Put $\lambda=f'(x_0)$. The point $x_0$ is attractive if $0<|\lambda|_p < 1$, {\it indifferent} if $|\lambda|_p = 1$,
and repelling if $|\lambda|_p > 1$.

The ball $U_r(x_0)$ is said to
be a {\it Siegel disk} if each sphere $S_{\r}(x_0)$, $\r<r$ is an
invariant sphere of $f(x)$, i.e., $f^n(x)\in S_{\r}(x_0)$ for all $n=0,1,2\dots$.  The
union of all Siegel disks with the center at $x_0$ is called {\it
	a maximum Siegel disk} and is denoted by $SI(x_0)$.

\section{Dynamics on the set of real numbers}

In this section we consider a real dynamical system associated with the function $f:\mathbb{R}\rightarrow \mathbb{R}$ defined by
\begin{equation}\label{kutubxona} f(x)=\frac{x+a}{bx+c}, \ \ a, b, c\in \R, \ \ b\ne 0, \ \ ab-c\neq 0,
\end{equation}
where $x\neq \hat x=-\frac{c}{b}$.

\begin{lemma}\label{ii}
	The function $f(x)$ given by (\ref{kutubxona}) is invertible on $\mathbb R\setminus \{\hat x\}$.
\end{lemma}
\begin{proof} From $ab-c\neq 0$ it follows that $f(x)\ne {1\over b}$ for any $x\in \mathbb R\setminus \{\hat x\}$.
Therefore, for inverse of $f$ we have
$$f^{-1}(x)= {a-cx\over bx-1}, \ \ x\ne {1\over b}.$$
\end{proof}
Introduce the set of `bad' points for $f$:
\begin{equation}\label{kP}
	\mathcal{P}=\{x\in \mathbb{R}: \exists n\in \mathbb{N}\cup {0}, f^n(x)=\hat x\},
\end{equation}
Using Lemma \ref{ii} we get
\begin{equation}\label{Pi}
	\mathcal{P}=\{f^{-n}(\hat x) : n\in \mathbb{N}\cup {0}\}.
\end{equation}

 \begin{figure}[h!]
	\includegraphics[width=9cm]{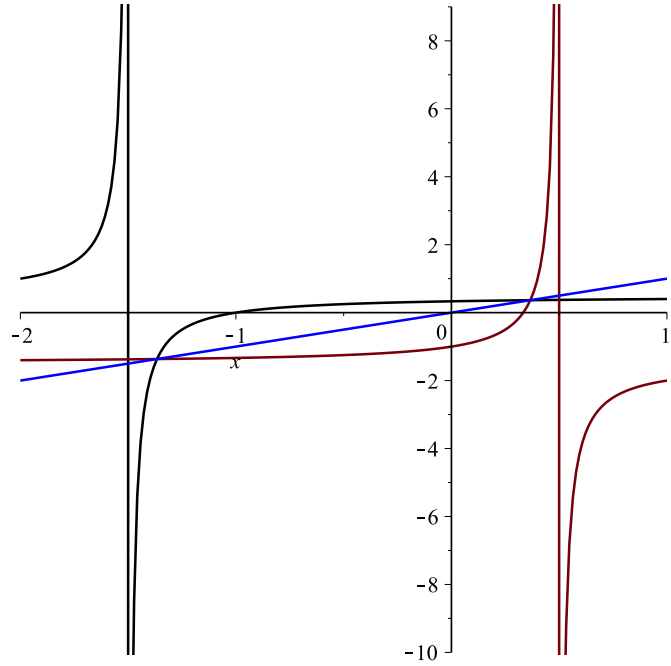}
	\caption{ The graphs of the function $f(x)$ (black),  $f^{-1}(x)$ (red) and $y=x$ (blue) for the case $a=1, b=2, c=3$. }\label{u1}
\end{figure}
%
%
%
 If $(c-1)^2+4ab\geq 0$ we define
\begin{equation}\label{ikkitanuqta} x_{1}=\frac{1-c+ \sqrt{(c-1)^2+4ab}}{2b}, \ \  x_{2}=\frac{1-c- \sqrt{(c-1)^2+4ab}}{2b} .
\end{equation}

Note that the set of fixed points of the function (\ref{kutubxona}) is
$${\rm Fix}(f)=\{x: f(x)=x\}=\left\{\begin{array}{lll}
\{x_1, x_2\}, \ \ \mbox{if} \ \ (c-1)^2+4ab>0\\[2mm]
\{x_1\}, \ \ \mbox{if} \ \ (c-1)^2+4ab=0\\[2mm]
\varnothing, \ \ \mbox{if} \ \ (c-1)^2+4ab<0.
\end{array}\right.$$

To study dynamical system (\ref{0a2}) for (\ref{kutubxona}) we use the following formula for $f^n(x)$, $n\geq 1$ (see\footnote{https://en.wikipedia.org/wiki/Iterated$_-$function}):
\begin{equation}\label{fn} f^n(x)=\frac{1}{b}+\frac{ab-c}{b}\cdot \left\{\begin{array}{ll}
		\frac{(bx-1+\alpha){\alpha}^{n-1}-(bx-1+\beta){\beta}^{n-1}}{(bx-1+\alpha){\alpha}^{n}-(bx-1+\beta){\beta}^{n}}, \ \ \mbox{if} \ \ \alpha\ne \beta\\[2mm]
	\frac{(bx-1)(n-1)+n\alpha}{\alpha[(bx-1)n+(n+1)\alpha]}, \ \ \mbox{if} \ \ \alpha=\beta,
	\end{array}\right.
\end{equation}
where
\begin{equation}\label{abe}\alpha\equiv \alpha(a,b,c)=\frac{1+c+\sqrt{(c-1)^2+4ab}}{2}, \ \ \beta\equiv \beta(a,b,c)=\frac{1+c-\sqrt{(c-1)^2+4ab}}{2}.
\end{equation}
\begin{rk}
	Note that formula (\ref{fn}) can be proven by induction. Moreover, it is true even for the case $(c-1)^2+4ab<0$.
\end{rk}
For $q\geq 1$ denote
\begin{equation}\label{kq}
K_q\equiv K_q(a,b,c)=\sum_{j=0}^{q-1}\alpha^{q-j-1}\beta^j.
\end{equation}

Let us give some examples of $K_q$:
$$K_q\equiv K_q(a,b,c)=\left\{\begin{array}{llll}
	1 \ \ \mbox{if} \ \ q=1\\[2mm]
	1+c \ \ \mbox{if} \ \ q=2\\[2mm]
	1+c+c^2+ab, \ \ \mbox{if} \ \ q=3\\[2mm]
	(1+c)(1+2ab+c^2)\ \ \mbox{if} \ \ q=4.\\[2mm]
\end{array}\right.
$$
\begin{thm}\label{th}
	Let function $f$ be given by parameters $(a,b,c)\in \mathbb R^3$ satisfying (\ref{kutubxona}).
	\begin{itemize}
		\item If  $K_q(a,b,c)\ne 0$ then ${\rm Per}_q(f)\setminus {\rm Fix}(f)=\emptyset.$
		\item If $K_q(a,b,c)= 0$ then any $x\in \mathbb R\setminus \mathcal P$ is $q$-periodic, i.e., ${\rm Per}_q(f)=\mathbb R\setminus \mathcal P$.
		\end{itemize}
\end{thm}
\begin{proof} A $q$-periodic point of $f$ (where $q\geq 1$) is a
 solution to equation $f^q(x)=x$. Since $f^q(x)$ has the form (see (\ref{fn})):
 $$f^q(x)={a_qx+b_q\over c_qx+d_q}$$
 we have
  $$f^q(x)=x \ \ \Leftrightarrow \ \ P_{2,q}(x)=c_qx^2+(d_q-a_q)x-b_q=0.$$

 Note that any solution of $f(x)=x$ (i.e. fixed point) is a solution to $f^q(x)=x$ too. Therefore $P_{2,1}(x)=bx^2+(c-1)x-a$ divides $P_{2,q}(x)$, i.e., there exists $K_q(a,b,c)$ (independent on $x$) such that
 \begin{equation}\label{pk}
 	P_{2,q}(x)=K_q(a,b,c)P_{2,1}(x).
 \end{equation}	
  Consequently, a $q$-periodic point different from fixed point exists if and only if $K_q(a,b,c)=0$. Moreover, if   $K_q(a,b,c)=0$ then $P_{2,q}(x)=0$ for any $x\in \mathbb R\setminus \mathcal P$. Now it remains to show that $K_q$ has the form (\ref{kq}). From (\ref{pk}) we get
  \begin{equation}\label{ck}
  	c_q=bK_q,  \ \ d_q-a_q=(c-1)K_q, \ \ b_q=aK_q
  	\end{equation}
  Moreover for $f^{q+1}$ we have
  $$f^{q+1}(x)={a_{q+1}x+b_{q+1}\over c_{q+1}x+d_{q+1}}={(a_q+bb_q)x+(aa_q+cb_q)\over (c_q+bd_q)x+(ac_q+cd_q)}.$$
 Hence
  \begin{equation}\label{cd}
  	\begin{array}{llll}
  	a_{q+1}=a_q+bb_q\\[2mm]
  	 b_{q+1}=aa_q+cb_q\\[2mm]
 c_{q+1}=c_q+bd_q\\[2mm]
  d_{q+1}=ac_q+cd_q,
 \end{array}
 \end{equation}
with initial conditions $a_1=1$, $b_1=a$, $c_1=b$ and $d_1=c$.

  By (\ref{ck}) from (\ref{cd}) we obtain
   \begin{equation}\label{Kc}
  	K_{q+1}=K_q+d_q,  \ \ d_{q+1}=abK_q+cd_q.
  \end{equation}
  Find $d_q$ from the first equation of (\ref{Kc}) and substituting it to the second one we get
   \begin{equation}\label{kk}
  	K_{q+2}-(c+1)K_{q+1}+(c-ab)K_q=0.
  \end{equation}
  The equation (\ref{kk}) has general solution
  \begin{equation}\label{ks}K_q=\left\{\begin{array}{ll}
  	A\alpha^q+B\beta^q, \ \ \mbox{if} \ \ \alpha\ne \beta\\[2mm]
  	(A+qB)\alpha^q,  \ \ \mbox{if} \ \ \alpha=\beta,
  	\end{array}
  	\right.\end{equation}
  where $\alpha$ and $\beta$ are defined in (\ref{abe}).

  	Since $K_1=1$ and $K_2=1+c$  from (\ref{ks}) one can find corresponding $A$ and $B$ and see that $K_q$ is as in (\ref{kq}).
 \end{proof}
Denote
$$\mathbb K_q=\{(a,b, c)\in \mathbb R^3: K_{q}(a,b,c)=0\}.$$
$$\mathbb K=\mathbb R^3\setminus \bigcup_{q=2}^{+\infty} \mathbb K_q.$$
\begin{thm} Let function $f$ be given by parameters $(a,b,c)\in \mathbb K$ satisfying (\ref{kutubxona}).  Then following equalities hold:
\begin{itemize}
\item[1)] If $(c-1)^2+4ab>0$   then for any $x\in \mathbb{R}\setminus ({\rm Fix}(f)\cup \mathcal P)$  $$\lim_{n\rightarrow\infty}f^n(x)=
 \left\{\begin{array}{lll}
x_2, \ \ {\rm if} \ \left|\frac{\alpha}{\beta}\right|<1\\[2mm]
x_1, \ \ {\rm if} \ \left|\frac{\alpha}{\beta}\right|>1.\end{array} \right.$$

\item[2)] If $(c-1)^2+4ab=0$ then for any $x\in \mathbb{R}\setminus ({\rm Fix}(f)\cup \mathcal P)$  $$\lim_{n\rightarrow\infty}f^n(x)=x_1={1-c\over 2b}.$$

\item[3)] If $(c-1)^2+4ab<0$ then for any $x\in \mathbb{R}\setminus ({\rm Fix}(f)\cup \mathcal P)$ the trajectory $\{f^n(x)\}_{n\geq 1}$  is dense in $\mathbb R$.
\end{itemize}
\end{thm}
\begin{proof} 1) Let $|\frac{\alpha}{\beta}|<1$. We can write $f^n(x)$ as follows:
$$f^n(x)=\frac{1}{b}\left[1+(ab-c)\cdot \frac{(bx-1+\alpha)\left(\frac{\alpha}{\beta}\right)^{n-1}-(bx-1+\beta)}{(bx-1+\alpha)\left(\frac{\alpha}{\beta}\right)^{n-1}\alpha-(bx-1+\beta)\beta}\right].$$
Since $\lim_{n\rightarrow \infty}\left(\frac{\alpha}{\beta}\right)^{n-1}=0$, we get $\lim_{n\rightarrow \infty}f^n(x)=\frac{1}{b}\left(1+\frac{ab-c}{\beta}\right)=x_2$.
The case $\left|\frac{\alpha}{\beta}\right|>1$ is proved similarly.

2) Follows from the second line of formula (\ref{fn}).

3) If $(c-1)^2+4ab<0$ then $ab<0$ and
$$1-2\sqrt{-ab}<c<1+2\sqrt{-ab}.$$
Consequently,
$$(1-\sqrt{-ab})^2<c-ab<(1+\sqrt{-ab})^2.$$
Moreover, it is easy to see that $\alpha$ and $\beta$ are complex
 numbers with
$$r:=|\alpha|=|\beta|=\sqrt{c-ab}.$$
\begin{equation}\label{theta}
\theta=\arg(\alpha)=-\arg(\beta)=\arctan\left({\sqrt{-(c-1)^2-4ab}\over c+1}\right).	
\end{equation}
This leads to the polar form and by de Moivre's formula we have:
$$\alpha^k =r^k[\cos (k\theta) +i\sin (k\theta)], \ \ \beta^k =r^k[\cos (k\theta) +i\sin (k\theta)], \ \ k\in \mathbb N.$$
Using these formulas from (\ref{fn}) we get
$$f^n(x)={(bx-1)\sin((n-1)\theta)+r\sin(n\theta)\over r[(bx-1)\sin(n\theta)+r\sin((n+1)\theta)]}$$
$$={\{(bx-1)\cos(\theta)+r\}\sin(n\theta)-(bx-1)\sin(\theta)\cos(n\theta)\over r[\{bx-1+r\cos(\theta)\}\sin(n\theta)+r\sin(\theta)\cos(n\theta)]}.$$
Denote
$$g(t)\equiv g(t, r, bx, \theta)={\{(bx-1)\cos(\theta)+r\}t-(bx-1)\sin(\theta)\over r[\{bx-1+r\cos(\theta)\}t+r\sin(\theta)]}$$
and $t_n=\tan(n\theta)$. Then $f^n(x)=g(t_n)$.
Thus for any $x\in \mathbb{R}\setminus ({\rm Fix}(f)\cup \mathcal P)$ we conclude that the set of limit points of sequence  $\{f^n(x)\}_{n\geq 1}$ can be given by the set of limit points of the sequence $\{t_n=\tan(n\theta)\}_{n\geq 1}$.
The following lemma gives the set of limit points of $\{t_n\}_{n\geq 1}$.

\begin{lemma}\label{bt} The following hold
	
	\begin{itemize}
		\item[1.] If $\theta\in \mathbb R$ is a rational number,
		then $\{t_n\}_{n\geq 1}$  is periodic.
		Moreover, if $\theta=s/q$ then the length of periodicity is equal to $q$.
		
		\item[2.]If $\theta\in \mathbb R$ is an irrational number,
		then $\{t_n\}_{n\geq 1}$  is dense in $\mathbb R$.
		In other words, for any $a\in \mathbb R$ there is a sequence $\{n_k\}_{k=1}^\infty$ of natural numbers such that
		$$\lim_{k\to\infty}t_{n_k}=a.$$
	\end{itemize}
\end{lemma}
{\bf Proof of lemma}. We shall use known (see \cite[Section 2.2]{Rmb})
theorem of mathematical billiards on circular table. On such a table a billiard trajectory is completely determined by the angle $\theta$ made with the circle.
Define $T_\theta: [0,1]\to [0,1]$ such
that $T_\theta(x) = x + \theta (mod \, 1)$, where $x\in [0,1]$
and $\theta\in \mathbb R$.
Here $\theta$
denotes the angle of rotation along the circle. The $n$ iterations of the map $T_\theta$ given as follows:
\begin{equation}\label{Tn}
	T^n_\theta(x) = x + n\theta (mod \, 1).
\end{equation}
By \cite[Theorem 2.1]{Rmb} the following are known
	\begin{itemize}
		\item[(i)] If $\theta\in \mathbb R$ is a rational number,
		then every orbit of the mapping $T_\theta(x)$ is periodic.
		Moreover, if $\theta=p/q$ then the length of periodicity is equal to $q$.
		
		\item[(ii)] If $\theta\in \mathbb R$ is an irrational number,	then for each $x\in [0,1]$ the sequence $\{T^n_\theta(x)\}_{n\geq 0}$ is dense in $[0,1]$.
	\end{itemize}
In order to prove lemma, consider a
billiard on a circle with radius 1, which corresponds to the
rotation number $\theta$ radian. Then sequence
$$t_0=0, t_1=\tan(\theta), t_2=\tan(2\theta), t_3=\tan(3\theta),...$$ on $\mathbb R$ corresponds to the trajectory
$0, \theta, 2\theta, 3\theta, ...$ of the billiard with the starting point 0. Therefore, lemma follows from properties (i) and (ii) mentioned above. Lemma is proved.

Now we continue proof of part 3) of theorem.
By assumption $(a,b,c)\in \mathbb K$  of theorem the quantity $\theta$ defined in (\ref{theta}) is an irrational number. Since $g(t)$ maps $\mathbb R$ on itself the part 3) follows from part 2 of Lemma \ref{bt}.
\end{proof}
\section{$f(x)$ on the set of complex $p$-adic numbers}

In this section we consider dynamical system associated with the function $f:\C_p\rightarrow \C_p$ defined as
\begin{equation}\label{fp}
	f(x)={x+a\over bx+c}, \ \ b\ne 0, \ \ c\ne ab, \ \ a,b,c\in \mathbb C_p,
\end{equation}
where $x\ne \hat x=-c/b$.

\subsection{Known results on $p$-adic dynamical system}
Following \cite{Mukhamedov3} we give known results about dynamical systems generated by (\ref{fp}). 

Since $\mathbb C_p$ is an algebraic closed field, the points \eqref{ikkitanuqta} are fixed for $p$-adic version of $f$ too.  

\subsubsection{Case: $|f'(x_1)|_p=1$.}

\begin{thm}\label{m3}\cite{Mukhamedov3} If
\begin{equation}\label{c1}
\left|\frac{b}{\sqrt{c-ab}}\right|_p<1 \ \ \mbox{and} \ \ 	|f'(x_i)|_p=\left|\frac{c-ab}{(bx_i+c)^2}\right|_p=1, \ \ i=1, 2.
\end{equation}
Then
$$SI(x_i)=U_{1+\varepsilon_c}(x_i), \ \ i=1, 2,$$
where $\varepsilon_c=\left|\frac{\sqrt{c-ab}}{b}\right|_p-1$.
\end{thm}
The following theorem gives relations between Siegel disks $SI(x_i)$, $i=1,2$.
\begin{thm}\cite{Mukhamedov3} If condition (\ref{c1})
is satisfied. Then
\begin{itemize}
\item[1.] If $\left|\frac{\sqrt{(c-1)^2+4ab}}{b}\right|_p\geq 1+\varepsilon_c$, then $SI(x_1)\cap SI(x_2)=\varnothing;$
\item[2.] otherwise $SI(x_1)=SI(x_2)$.
\end{itemize}
\end{thm}

\subsubsection{Case: $|f'(x_1)|_p\neq 1$}
Denote
$$\delta_2=\left|{bx_1+c\over b}\right|-1.$$

\begin{thm}\label{cupteorama}\cite{Mukhamedov3} If 
\begin{equation}\label{c2} \left|\frac{c-ab}{(bx_1+c)^2}\right|_p<1 \ \ \mbox{and} \ \ 
\left|\frac{b}{bx_1+c}\right|_p<1.
\end{equation} Then $$\bigcup_{\delta:0\leq \delta\neq 1+\delta_2}S_\delta(x_1)\subset \mathcal A(x_1).$$
\end{thm}

\begin{thm}\label{spminus}\cite{Mukhamedov3} Let the condition (\ref{c2}) be satisfied. Then $$\mathcal A(x_1)=\C_p\setminus\{\hat x, x_2\}.$$
\end{thm}

\subsection{New approach for $p$-adic case}

Let us consider the function (\ref{fp}) and use methods developed in \cite{ARS}, \cite{RS} -\cite{S} to study dynamical systems of (\ref{fp}).

\subsubsection{The unique fixed point}
If $D=(c-1)^2+4ab=0$, then (\ref{fp}) function has unique fixed point: $x_0=\frac{1-c}{2b}$. Now we find $|f'(x_0)|_p$:
\begin{eqnarray*}
	|f'(x_0)|_p=\left|\frac{c-ab}{(bx_0+c)^2}\right|_p=\left|\frac{c-ab}{(b\cdot \frac{1-c}{2b}+c)^2}\right|_p=\left|\frac{4(c-ab)}{(1+c)^2}\right|_p=\\=\left|\frac{4c-4ab}{(1+c)^2}\right|_p=\left|\frac{4c+(c-1)^2}{(1+c)^2}\right|_p=\left|\frac{(1+c)^2}{(1+c)^2}\right|_p=1.
\end{eqnarray*}

So, $x_0$ is an indifferent fixed point. 

Using equality $(c-1)^2+4ab=0$ and $x_0=\frac{1-c}{2b}$ 
we get the following:
\begin{equation}\label{ayirma}
|f(x)-x_0|_p=|x-x_0|_p\cdot \left|\frac{\frac{c+1}{2b}}{(x-x_0)+\frac{c+1}{2b}}\right|_p.
\end{equation}

Denote the following
\begin{equation}\label{kattaP}
\mathcal{P}_p=\{x\in \mathbb{C}_p: \exists n\in \mathbb{N}\cup {0}, f^n(x)=\hat x\},
\end{equation}
$$\alpha=\left|\frac{c+1}{2b}\right|_p.$$

For $\alpha\geq 0$ define the function $\psi_\alpha:[0,+\infty)\rightarrow [0,+\infty)$ by
$$\psi_\alpha(r)=\left\{\begin{array}{lll}
\ r, \ \ \ \ {\rm if }\ \  r < \alpha\\[2mm]
\alpha^{*}, \ \ \ {\rm if} \ \ r=\alpha\\[2mm]
\alpha, \  \  \ \ {\rm if} \ \ r> \alpha
\end{array} \right.,$$

where $\alpha^{*}$ is a positive number with $\alpha^{*}\geq \alpha$.

\begin{lemma}\label{1} If $x\in S_r(x_0)\setminus\mathcal{P}_p$, then for the function (\ref{fp}) the following formula holds
$$|f^n(x)-x_0|_p=\psi_\alpha^{n}(r).$$
\end{lemma}
\begin{proof} Since $|x-x_0|_p=r$, $\left|\frac{c+1}{2b}\right|_p=\alpha$, using formula (\ref{ayirma}) and the strong triangle inequality of the $p$-adic norm, we get the following
\begin{equation}\label{ayirma2}
|f(x)-x_0|_p=|x-x_0|_p\cdot \left|\frac{\frac{c+1}{2b}}{(x-x_0)+\frac{c+1}{2b}}\right|_p=\varphi_\alpha(r)=\left\{\begin{array}{ccc}
\ r, \ \ \ \ \  {\rm if }\ \  r < \alpha\\[2mm]
\alpha^{*}, \ \ \ {\rm if} \ \ r=\alpha\\[2mm]
\alpha, \  \  \ \ {\rm if} \ \ r> \alpha.
\end{array} \right.
\end{equation}
Now consider the case $n=2.$ Since $|f(x)-x_0|_p=\psi_\alpha(r)$ (by (\ref{ayirma2})), we obtain $$|f^2(x)-x_0|_p=|f(x)-x_0|_p\cdot \left|\frac{\frac{c+1}{2b}}{(f(x)-x_0)+\frac{c+1}{2b}}\right|_p=\psi_\alpha(\psi_\alpha(r))=\left\{\begin{array}{ccc} \ \psi_\alpha(r), \ \ \ \ {\rm if }\ \ \psi_\alpha(r)<\alpha\\[2mm]
\geq \alpha^{*}, \ \ \ {\rm if} \ \ \psi_\alpha(r)=\alpha\\[2mm]
\alpha, \ \ \ \ {\rm if} \ \ \psi_\alpha(r)>\alpha.
\end{array} \right.$$

Iterating this argument for any $n\geq 1$ and any $x\in S_r(x_0)\setminus \mathcal{P}$, we obtain the following formula $$|f^n(x)-x_0|_p=\psi_\alpha^n(r).$$
\end{proof}

\begin{lemma}\label{ok} The function $\psi_\alpha$ has the following properties
\begin{itemize}
\item[1.] ${\rm Fix}(\psi_\alpha)=\{r: 0\leq r<\alpha\}\cup \{\alpha: \rm{if} \ \alpha^{*}=\alpha\}$.
\item[2.] If $r>\alpha$ then $\psi_\alpha(r)=\alpha$.
\item[3.] If $r=\alpha$ and $\alpha^{*}>\alpha$, then $\psi_\alpha^{2}(\alpha)=\alpha$.
\end{itemize}
\end{lemma}
\begin{proof}
\begin{itemize}
\item[1)] This is a simple observation of the equation $\psi_\alpha(r)=r.$
\item[2)] If $r>\alpha$, then by definition of function $\psi_\alpha$, we have $$\psi_\alpha(r)=\alpha.$$
\item[3)] If $r=\alpha$, then by definiton of function $\psi_\alpha$, we have $\psi_\alpha(r)=\alpha^{*}$ and since $\alpha^{*}>\alpha$ by part 2 of this lemma we get $\psi_\alpha(\alpha^{*})=\psi_\alpha(\psi_\alpha(r))=\psi_\alpha^{2}(r)=\alpha$.

\end{itemize}
\end{proof}

Denote
$$\alpha^{*}(x)= |f(x)-x_0|_p,  \  \ {\rm if} \ \ x\in S_\alpha(x_0).$$

By applying Lemma \ref{1} and \ref{ok} we get the following properties of the $p$-adic dynamical system complied by the function (\ref{fp}).

\begin{thm}\label{t1} The $p$-adic dynamical system is generated by the function (\ref{fp}) has the following properties:
\begin{itemize}
\item[1.] $SI(x_0)=U_\alpha(x_0)$.
\item[2.] If $r>\alpha$ and $x\in S_r(x_0)$, then $f(x)\in S_\alpha(x_0)$.
\item[3.] Let $x\in S_\alpha(x_0)\setminus \mathcal{P}_p$.
\begin{itemize}
\item[3.1)] If $\alpha^{*}(x)=\alpha$, then $f(x)\in S_\alpha(x_0)$.
\item[3.2)] If $\alpha^{*}(x)>\alpha$, then $f^2(x)\in S_\alpha(x_0)$.
\end{itemize}
\end{itemize}
\end{thm}
\begin{rk}
	We note that part 1 of Theorem \ref{t1} coincides with Theorem \ref{m3}, because from $(c-1)^2+4ab=0$ it follows that $\alpha=1+\varepsilon_c$. Parts 2 and 3 mean that any point taken from the outside of $U_\alpha(x_0)$ after up to two-step of iterations of (\ref{fp}) comes in the sphere $S_\alpha(x_0).$
\end{rk}
\subsubsection{Two fixed points} If $D=(c-1)^2+4ab\neq 0$ then we have two fixed points $x_1,x_2$.
For any $x\in \mathbb{C}_p$, $x\neq \hat x$, by simple calculation we get
\begin{equation}\label{simple}
|f(x)-x_i|_p=|x-x_i|_p\cdot \frac{|\alpha(x_i)|_p}{|\beta(x_i)+x-x_i|_p},\ \ \ i=1,2,
\end{equation}
where $$\alpha(x)=\frac{1-bx}{b}, \ \ \ \beta(x)=\frac{bx+c}{b}.$$
Define $$\alpha_i=|\alpha(x_i)|_p, \ \ \ \beta_i=|\beta(x_i)|_p, \ \ \ i=1,2.$$

For $\beta>0$ define the function $\varphi_{\alpha,\beta}:[0,+\infty)\rightarrow [0,+\infty)$ by
$$\varphi_{\alpha,\beta}(r)=\left\{\begin{array}{lll}
\frac{\alpha}{\beta}r, \ \ \ \ {\rm if }\ \  r < \beta\\[2mm]
\beta^{*}, \  \  \ \ {\rm if} \ \ r=\beta\\[2mm]
\alpha, \ \ \ \ \ {\rm if} \ \ r>\beta
\end{array} \right.$$
where $\beta^{*}$ is given positive number with $\beta^{*}\geq \alpha.$

For $\beta=\alpha\geq 0$ we have $D=(c-1)^2+4ab=0$, this is the above studied case, when function has the unique fixed point.
Therefore, we assume $\alpha\ne \beta$.

Using formula (\ref{simple}) we easily get the following:
\begin{lemma}\label{filemma} If $x\in S_r(x_i)$, then the following formula holds
$$|f^n(x)-x_i|_p=\varphi^n_{\alpha_i,\beta_i}(r), \ \ n\geq 1, \ \ i=1,2.$$
\end{lemma}

Thus the $p$-adic dynamical system $f^n(x)$, $n\geq 1$, $x\in \mathbb{C}_p$, $x\neq \hat x$ is related to the real dynamical system generated by $\varphi_{\alpha,\beta}$. Now we are going to study this (real) dynamical system.

\begin{lemma}\label{sistemalemma} The dynamical system generated by $\varphi_{\alpha,\beta}(r)$, has the following properties:
\begin{itemize}
\item[1.] ${\rm Fix}(\varphi_{\alpha,\beta})=\{0\}\cup\left\{\begin{array}{ll}
\{\beta^{*}:{\rm if } \  \beta=\beta^{*}\} \ \ {\rm for} \ \alpha<\beta\\[2mm]
\{\alpha\} \ \ {\rm for} \ \alpha>\beta.\end{array} \right.$
\item[2.] $\lim_{n\rightarrow\infty}\varphi^{n}_{\alpha,\beta}(r)=\left\{\begin{array}{ll}
0, \ \ {\rm for \ any} \ r\geq 0, \ \ {\rm if} \ \alpha<\beta\\[2mm]
\alpha, \ \ {\rm for \ any} \ r\geq 0, \ \ {\rm if} \ \alpha>\beta.\end{array}\right.$
\end{itemize}
\end{lemma}
\begin{proof}
\begin{itemize}
\item[1.] This is the result of a simple analysis of the equation $\varphi_{\alpha,\beta}(r)=r$.
\item[2.] Since $\varphi_{\alpha,\beta}(r)$ is a piecewise linear function, the proof consists of simple computations, using the graph by varying parameters $\alpha,\beta$.
\end{itemize}
\end{proof}

Using Lemmas \ref{filemma} and \ref{sistemalemma} we obtain the following:
\begin{thm}\label{thm2} If $x\in \C_p\setminus (\mathcal P_p\cup\{ x_2\})$ (resp. $x\in \C_p\setminus(\mathcal P_p\cup\{ x_1\})$), then the $p$-adic dynamical system generated by $f$ has the following properties:
\begin{itemize}
\item[1.] $\lim_{n\rightarrow \infty}f^n(x)=x_1$ for $\alpha_1<\beta_1$.
\item[] (resp. $\lim_{n\rightarrow \infty}f^n(x)=x_2$ for $\alpha_2<\beta_2$.)
\item[2.] $\lim_{n\rightarrow \infty}f^n(x)\in S_{\alpha_i}(x_i)$ for $\alpha_i>\beta_i$.
\end{itemize}
\end{thm}
\begin{rk} We note that part 1 of Theorem \ref{thm2} is a reformulation of Theorem \ref{spminus}. Moreover, part 2 gives additional result, which was not mentioned in \cite{Mukhamedov3}. 
\end{rk}
 
\subsection{Application of formula (\ref{fn}) for $p$-adic case}
We note that in the previous subsection our method is reduction the $p$-adic dynamical system to the real one, where the real values are radius of spheres with center at a fixed point. 
In this subsection we apply formula (\ref{fn}) for the 
$p$-adic function (\ref{fp}), this allows to solve the main problem of the dynamical system by direct computation of limit points.  

We note that the following $p$-adic version of Theorem \ref{th} is true
\begin{thm}\label{thp}
	Let function $f$ be given by parameters $(a,b,c)\in \mathbb C_p$ satisfying (\ref{fp}).
	\begin{itemize}
		\item If  $K_q(a,b,c)\ne 0$ then ${\rm Per}_q(f)\setminus {\rm Fix}(f)=\emptyset.$
		\item If $K_q(a,b,c)= 0$ then any $x\in \mathbb C_p\setminus \mathcal P_p$ is $q$-periodic.
	\end{itemize}
\end{thm}
Denote
$$\mathbb K^{(p)}_q=\{(a,b, c)\in \mathbb C_p: K_{q}(a,b,c)=0\}.$$
$$\mathbb K^{(p)}=\mathbb C_p\setminus \bigcup_{q=2}^{+\infty} \mathbb K^{(p)}_q.$$
Below we consider $\alpha$ and $\beta$ defined in (\ref{abe}) as $p$-adic numbers given in $\mathbb C_p$.
\begin{thm} Let function $f$ be given by parameters $(a,b,c)\in \mathbb K^{(p)}$ satisfying (\ref{fp}).  Then following equalities hold:
	\begin{itemize}
		\item[a)] If $(c-1)^2+4ab\ne 0$   then for any $x\in \mathbb C_p\setminus ({\rm Fix}(f)\cup \mathcal P_p)$  $$\lim_{n\rightarrow\infty}f^n(x)=
		\left\{\begin{array}{lll}
			x_2, \ \ {\rm if} \ \left|{\alpha\over \beta}\right|_p<1\\[2mm]
			x_1, \ \ {\rm if} \ \left|{\alpha\over \beta}\right|_p>1.\end{array} \right.$$
		
		\item[b)] If $(c-1)^2+4ab=0$ then for any $x\in \mathbb C_p\setminus ({\rm Fix}(f)\cup \mathcal P_p)$  $$\lim_{n\rightarrow\infty}f^n(x)=x_1={1-c\over 2b}.$$
		
		\end{itemize}
\end{thm}
\begin{proof} Similar to the proof of Theorem \ref{th}, where one has to replace $|\cdot|$ by $|\cdot|_p$. 
\end{proof}
\begin{rk} Since $\mathbb C_p$ is algebraic closed, the function (\ref{fp}) always has fixed points, therefore an analogue of part 3 of Theorem \ref{th} does not appear in Theorem \ref{thp}.
\end{rk}
\begin{rk} Note that there are many papers devoted to 
dynamical systems generated by function $f(z)={az+b\over cz+d}$, (where $z\in \mathbb C$, i.e., a complex number) which is called a M\"obius transformation of the complex plane (see\footnote{https://en.wikipedia.org/wiki/M\"obius$_-$transformation} for detailed properties of this function). 
 In  \cite[Chapter 1]{Be}, by using (\ref{fn}), it is shown that
 \begin{itemize}
 	\item if $f$ has unique fixed point $\xi$, then for $f^n(z)$ converges to $\xi$ for any $z\in \mathbb C$.
 	
 \item if $f$ has two fixed points, then either the $f^n(z)$ converges to one of fixed points, or they move cyclically through a finite set of points (periodic points), or they form a dense subset of some circle. 
  \end{itemize}
\end{rk}

\end{document}